\documentclass[preprint, 11pt, english]{elsarticle}
\usepackage{amsmath}
\usepackage{latexsym, amssymb}
\usepackage{amsthm}
\usepackage{txfonts}
\usepackage[small,nohug,heads=littlevee]{diagrams} 

\newtheorem{thm}{Theorem}[section] 

\newtheorem{cor}[thm]{Corollary}

\newtheorem{prop}[thm]{Proposition}

\theoremstyle{definition}
\newtheorem{rem}[thm]{Remark}
\newtheorem{exmpl}[thm]{Example}
\newtheorem{defn}[thm]{Definition}
\newtheorem{algo}[thm]{Algorithm}

\newcommand\operA[2]{{\if!#2!\operatorname{#1}\else{\operatorname{#1}_{#2}^{\phantom{I}}}\fi}} 

\newcommand\Cref[1]{{Corollary~\ref{#1}}}%

\def\tr{{\operatorname{Tr}}}

\def\norm{{\operatorname{Norm}}}

\newcommand{\Trace}[1][]{\if!#1!\operatorname{Tr}\else{\operatorname{Tr}_{#1}^{\phantom{I}}}\fi} 

\long\def\forget#1\forgotten{{}} %

\def\({\left(}
\def\){\right)}

\newif\iffurther
\furtherfalse

\newif\ifXY 
\XYtrue     
\ifXY

\input xy
\input xyidioms.tex
\usepackage{xy}
\xyoption{all} %
\fi 

\usepackage{babel}

\journal{Journal of Algebra and its Applications}

\begin{document}

\begin{frontmatter}

\title{Polynomial Equations over Octonion Algebras}
\author{Adam Chapman}
\ead{adam1chapman@yahoo.com}
\address{Department of Computer Science, Tel-Hai Academic College, Upper Galilee, 12208 Israel}

\begin{abstract}
In this paper we present a complete method for finding the roots of all polynomials of the form $\phi(z)=c_n z^n+c_{n-1} z^{n-1}+\dots+c_1 z+c_0$ over a given octonion division algebra. When $\phi(z)$ is monic we also consider the companion matrix and its left and right eigenvalues and study their relations to the roots of $\phi(z)$, showing that the right eigenvalues form the conjugacy classes of the roots of $\phi(z)$ and the left eigenvalues form a larger set than the roots of $\phi(z)$.
\end{abstract}

\begin{keyword}
Polynomial Equations, Division Algebras, Octonion Algebras, Companion Matrix, Right Eigenvalues, Left Eigenvalues
\MSC[2010] primary 17D05; secondary 15A18, 15B33
\end{keyword}

\end{frontmatter}

\section{Introduction}

The question of finding the roots of a given (monic) standard polynomial $\phi(z)=z^n+c_{n-1} z_{n-1}+\dots+c_1 z+c_0$ over any quaternion division algebra $Q$ with center $F$ was fully solved in \cite{ChapmanMachen:2017}:
the polynomial has an assigned ``companion polynomial" $\Phi(z)$ whose degree is $2n$ and its coefficients live in $F$, which is also the companion polynomial of the embedding of the companion matrix $C_\phi$ of $\phi(z)$ into $M_{2n}(K)$ where $K$ is an arbitrary maximal subfield of $Q$.
The left eigenvalues of $C_\phi$ coincide with the roots of $\phi(z)$, and the right eigenvalues of $C_\phi$ coincide with the roots of $\Phi(z)$.
The roots of $\Phi(z)$ group into (up to $n$) complete conjugacy classes.
For each such conjugacy class, either the entire class consists of roots of $\phi(z)$, or it contains exactly one root of $\phi(z)$.
Earlier papers on this subject include \cite{JanovskaOpfer2010} (solving equations over the real quaternion algebra), \cite{Abrate2009} and \citep{Chapman2015} (solving quadratic equations over arbitrary quaternion algebras), and \cite{WangZhang} (solving monic quadratic equations over the real octonion algebra).

The aim of this paper is to extend these results to octonion division algebras.
A part of the motivation comes from recent results in physics that translate physical problems to equations over octonions and more general Cayley-Dickson algebras, see for example \cite{Mironov:2017}.
We consider \emph{standard polynomials} over an octonion division algebra $A$. These are polynomials with coefficients appearing only on the left-hand side of the variable:
$
\phi(z)=c_n z^n+c_{n-1} z^{n-1}+\dots+c_1 z+c_0
$
where $c_i\in A$. For any $\lambda \in A$ the substitution of $\lambda$ in $\phi(z)$ is defined to be $c_n \lambda^n+c_{n-1} \lambda^{n-1}+\dots+c_1 \lambda+c_0$ and is denoted by $\phi(\lambda)$.
By a \emph{root} of a standard polynomial $\phi(z)$ over $A$ we mean an element $\lambda \in A$ satisfying $\phi(\lambda)=0$.
We denote by $R(\phi)$ the set of roots of $\phi(z)$.
We define the companion polynomial $\Phi(z)$ and the companion matrix in the same manner as in the quaternionic case. We show that the roots of $\Phi(z)$ are the conjugacy classes of the roots of $\phi(z)$.
We prove that for each class in $R(\Phi)$ either the entire class is in $R(\phi)$ or it contains a unique element form $R(\phi)$. We prove that the right eigenvalues of $C_\phi$ are exactly the roots of $\Phi(z)$ and also describe its left eigenvalues.
Unlike the quaternionic case, the left eigenvalues of $C_\phi$ turn out to be a larger set than the roots of $\phi(z)$ and we provide an example of this phenomenon.

Note that our definition of a standard polynomial places the coefficients on the left-hand side of the variable, but clearly the same methods can be applied to solving polynomials with coefficients appearing on the right-hand side.
Given a standard polynomial $\phi(z)=c_n z^n+\dots+c_1 z+c_0$ over an octonion algebra $A$, we define its ``mirror" polynomial to be $\tilde{\phi}(z)=z^n c_n+\dots+z c_1+c_0$.

\section{Octonion Algebras}

Given a field $F$, a quaternion algebra $Q$ over $F$ is a central simple $F$-algebra of degree 2 (i.e., dimension 4 over $F$).
When $\operatorname{char}(F) \neq 2$, it has the structure
$$Q=F\langle i,j : i^2=\alpha, j^2=\beta, i j=-j i \rangle$$
for some $\alpha,\beta\in Q^\times$,
and when $\operatorname{char}(F)=2$, it has the structure
$$Q=F\langle i,j : i^2+i=\alpha, j^2=\beta, j i+i j=j \rangle$$
for some $\alpha \in F$ and $\beta \in F^\times$.
The algebra $Q$ is endowed with a symplectic involution mapping $a+bi+cj+dij$ to $a-bi-cj-dij$ when $\operatorname{char}(F)\neq 2$ and to $a+b+bi+cj+dij$ when $\operatorname{char}(F)=2$.

An \emph{Octonion} algebra over $F$ is an algebra $A$ of the form $A=Q \oplus Q\ell$ where $Q$ is a quaternion algebra over $F$, and the multiplication table is given by
$(p+q \ell)\cdot (r+s \ell)=pr+\ell^2 \bar{s}q+(sp+q\bar{r})\ell$ where $\bar{\ }$ stands for the symplectic involution on $Q$ and $\ell^2 \in F^\times$.
This involution extends to $A$ by the formula $\overline{p+q \ell}=\overline{p}-q \ell$.
The octonion algebra is endowed with a quadratic norm form $\norm : A \rightarrow F$ defined by $\norm(x)=\overline{x} \cdot x$ and a linear trace form $\tr : A \rightarrow F$ defined by $\tr(x)=x+\overline{x}$. Every two elements in $A$ live inside a quaternion subalgebra, unless $\operatorname{char}(F)=2$, in which case the two elements can also live inside a purely inseparable bi-quadratic field extension of $F$ inside $A$, for example the elements $j$ and $\ell$ in the construction above. 
In particular, the algebra is alternative. The algebra is a division algebra if and only if its norm form is anisotropic. For further reading on octonion algebras see \cite{SpringerVeldkamp} and \cite{BOI}.

\section{The Companion Polynomial}

The goal of this section is to give a deterministic algorithm for finding all the roots of a given standard polynomial over an octonion division algebra.

\begin{rem}
The relation $g \sim g' \Leftrightarrow \exists h \in A^\times : hgh^{-1}=g'$ is an equivalence relation for elements of a given octonion algebra $A$.
\end{rem}

\begin{proof}
It is enough to show that $g \sim g'$ if and only if $\tr(g)=\tr(g')$ and $\norm(g)=\norm(g')$.
Write $T=\tr(g)$ and $N=\norm(g)$. Both live inside $F$.
Then $g^2-Tg+N=0$.
Since the octonion algebra is alternative and $T,N \in F$, we can conjugate this equation by $h$ and obtain
$(hgh^{-1})^2-T(hgh^{-1})+N=0$, which means that the trace and norm of $hgh^{-1}$ are $T$ and $N$, resp.
In the opposite direction, suppose $g$ and $g'$ have the same trace and norm. If they live inside a quaternion subalgebra then they are conjugates in that subalgebra, and if the live inside a purely inseparable bi-quadratic field extension of $F$ then they must be equal.
\end{proof}

\begin{defn}
We define the ``companion polynomial" $\Phi(z)$ of a given polynomial $\phi(z)=c_n z^n+\dots+c_1 z+c_0$ over an  octonion algebra $A$ to be 
$$\Phi(z)=b_{2n} z^{2n}+\dots+b_1 z+b_0$$
with the coefficients defined in the following way:
for each $k \in \{0,\dots,2n\}$, if $k$ is odd then $b_k$ is the sum of all $\tr(\bar{c_i}c_j)$ with $0\leq i < j \leq n$ and $i+j=k$, and if $k=2m$ is even then $b_k$ is the sum of all $\tr(\overline{c_i}c_j)$ with $0\leq i < j \leq n$ and $i+j=k$ plus the element $\norm(c_m)$. (Recall that $\tr(\bar{c_i}c_j)=\bar{c_i}c_j+\bar{c_j}c_i$ and $\norm(c_m)=\overline{c_m}c_m$.)
\end{defn}

\begin{thm}\label{solvingequations}
Let $\phi(z)=c_n (z^n)+\dots+c_1 z+c_0$ be a standard polynomial over an octonion division algebra $A$ over a field $F$ with companion polynomial $\Phi(z)$. Then $R(\Phi) \supseteq R(\phi)$.
\end{thm}

\begin{proof}
By the \cite[Lemma 1.3.3]{SpringerVeldkamp}, for every $z \in A$ and $i,j \in \{0,\dots,n\}$, $\norm(c_i) z^{2i}=\overline{c_i}(c_i z^{2i})$ and $\tr(\overline{c_i}c_j) z^{i+j}=\overline{c_i}(c_j z^{i+j})+\overline{c_j}(c_i z^{i+j})$.
Therefore
$$\Phi(z)=\sum_{i=0}^n \overline{c_i} (\sum_{j=0}^n c_j (z^{i+j}))=\sum_{i=0}^n \overline{c_i} (\phi(z) z^i).$$
Consequently, if $\phi(\lambda)=0$ for a certain $\lambda \in A$, then also $\Phi(\lambda)=0$.
\end{proof}

All the coefficients of $\Phi(z)$ are central, i.e. belong to $F$ (because they are sums of traces and norms of elements in $A$). Therefore, the roots of $\Phi(z)$ depend only on their norm and trace, i.e. the set $R(\Phi)$ is a union of conjugacy classes.

\begin{thm}\label{classes}
Given companion polynomial $\Phi(z)$ of $\phi(z)$, the set $R(\Phi)$ is the union of the conjugacy classes of the elements of $R(\phi)$. Each such class is either fully contained in $R(\phi)$ or has exactly one representative there.
\end{thm}

\begin{proof}
\sloppy Every $z\in A$ with $\tr(z)=T$ and $\norm(z)=N$ satisfies $z^2-Tz+N=0$.
Therefore, by plugging in $z^2=Tz-N$ in $\phi(z)$, we obtain $\phi(z)=E(N,T) z+G(N,T)$ for some polynomials $E(N,T)$ and $G(N,T)$ in the central variables $N$ and $T$.
Write $e_i(N,T)$ and $g_i(N,T)$ for the polynomials satisfying $z^i=e_i(N,T) z+g_i(N,T)$.
Note that $e_i(N,T)$ and $g_i(N,T)$ are central for any $N,T \in F$, and therefore we can treat them as central elements in the computations.
Then $E(N,T)=\sum_{i=0}^n c_i e_i$ and $G(N,T)=\sum_{i=0}^n c_i g_i$.
Now, $\Phi(z)=\sum_{i=0}^n \overline{c_i} (\phi(z) z^i)=\sum_{i=0}^n \overline{c_i} ((E(N,T) z+G(N,T)) (e_i(N,T) z+g_i(N,T)))=\sum_{i=0}^n e_i(N,T) \overline{c_i} (E(N,T) z^2)+\sum_{i=0}^n e_i(N,T) \overline{c_i} (G(N,T) z)+\sum_{i=0}^n g_i(N,T) \overline{c_i} (E(N,T) z)+\sum_{i=0}^n g_i(N,T) \overline{c_i} G(N,T)=\overline{E(N,T)} E(N,T) z^2+\overline{E(N,T)} G(N,T) z+\overline{G(N,T)} E(N,T) z+\overline{G(N,T)} G(N,T)=\norm(E(N,T)) z^2+\tr(\overline{E(N,T)} G(N,T)) z+\norm(G(N,T))$.

Consider an element $z_0 \in R(\Phi)$ with norm $N_0$ and trace $T_0$.
If $E(N_0,T_0)=0$ then we have $0=\Phi(z_0)=\norm(G(N_0,T_0))$, and so $G(N_0,T_0)=0$. This means that the equality $E(N_0,T_0) \lambda+G(N_0,T_0)=0$ holds for all $\lambda \in [z_0]$, and hence the equality $\phi(\lambda)=0$ holds for all $\lambda$ in the conjugacy class of $z_0$. 

Suppose $E(N_0,T_0) \neq 0$.
Since an element $\lambda$ of trace $T_0$ and norm $N_0$ satisfies $\phi(\lambda)=E(N_0,T_0) \lambda+G(N_0,T_0)$,
it is a root of $\phi(z)$ if and only if it is the unique solution to the equation $E(N_0,T_0) \lambda+G(N_0,T_0)=0$.
What is left then in order to prove that $R(\phi) \cap [z_0]=\{\lambda\}$ is to show that the unique solution $\lambda$ to the equation $E(N_0,T_0) \lambda+G(N_0,T_0)=0$ has indeed trace $T_0$ and norm $N_0$,
and indeed, this $\lambda$ satisfies $\norm(E(N_0,T_0)) \lambda^2+\tr(\overline{E(N_0,T_0)} G(N_0,T_0)) \lambda+\norm(G(N_0,T_0))=\overline{E(N_0,T_0)} (E(N_0,T_0) \lambda+G(N_0,T_0)) \lambda+\overline{G(N_0,T_0)} (E(N_0,T_0) \lambda+G(N_0,T_0))=0$, which means that $\lambda$ satisfies the same quadratic characteristic equation over $F$ as $z_0$, and so $\lambda$ is in the same conjugacy class as $z_0$, i.e. has norm $N_0$ and trace $T_0$.
\end{proof}

The previous two theorems give a complete algorithm for finding all the roots of an octonion polnyomial:
\begin{algo}
\sloppy One needs first to solve the equation $\Phi(z)$ over the algebraic closure of $F$. Each root $z_0$ lives in a field extension $K$ of $F$.
For $z_0$ to be in the same conjugacy class as an element of $R(\phi)$, $K$ must be $F$-isomorphic to a subfield of $A$, and therefore $[K:F]$ is either 2 or 1.
If it is, then the conjugacy class of $z_0$ in $A$ is in $R(\Phi)$, and then either $E(N_0,T_0)=0$ and then the entire class of $[z_0]$ is in $R(\phi)$, or $-E(N_0,T_0)^{-1}G(N_0,T_0)$ is the unique representative of $[z_0]$ in $R(\phi)$ where $N_0$ and $T_0$ are the norm and trace of $z_0$.
\end{algo}

\begin{exmpl}
Consider the real octonion algebra $A=\mathbb{O}$ with generators $i,j,\ell$, and the polynomial $\phi(z)=i z^2+j z+\ell$. The companion polynomial is $\Phi(z)=z^4+z^2+1$, and it has roots in the conjugacy classes of $\{z \in A: \norm(z)=1, \tr(z)=1\}$ and $\{z \in A: \norm(z)=1, \tr(z)=-1\}$.
For $\norm(z)=1, \tr(z)=1$, we have $z^2=z-1$, and so the equation $\phi(z)=0$ reduces to $(i+j) z+\ell-i=0$, which means that $(i+j)^{-1}(i-\ell)=\frac{1}{2}(1+ij+i\ell+j\ell)$ is the unique representative of its conjugacy class in $R(\phi)$.
For $\norm(z)=1, \tr(z)=-1$, we have $z^2=-z-1$, and so the equation $\phi(z)=0$ reduces to $(-i+j) z+\ell-i=0$, which means that $(-i+j)^{-1}(i-\ell)=\frac{1}{2}(-1+ij-i\ell+j\ell)$ is the unique representative of its conjugacy class in $R(\phi)$.
\end{exmpl}

\section{The Companion Matrix and its Left Eigenvalues}

Suppose $A$ is an octonion division algebra.
Let $\phi(z)=z^n+c_{n-1} z^{n-1}+\dots+c_0$ be a monic standard polynomial with coefficients $c_0,\dots,c_{n-1}$ in $A$. We want to associate the roots of $\phi(z)$ with left and right eigenvalues of the companion matrix, given by
\begin{center}
$C_\phi=\left(\begin{matrix}
0 & 1 & 0 & \dots & 0\\
0 & 0 & 1 & & 0\\
\vdots & & & \ddots & \\
0 &0 & \dots & 0& 1\\
-c_0 & -c_1 & \dots &-c_{n-2} & -c_{n-1}\rlap{~.}
\end{matrix}\right)$
\end{center}
We define the $\gamma$-twist of $\phi(z)$ to be the polynomial
$$\phi_\gamma(z)=\gamma^{-1} z^n+(\gamma^{-1} c_{n-1}) z^{n-1}+\dots+(\gamma^{-1} c_1) z+\gamma^{-1} c_0.$$
A left (or right) eigenvalue of $C_\phi$ is an element $\lambda \in A$ which satisfies $C_\phi v=\lambda v$ ($C_\phi v=v \lambda$) for some nonzero column vector $v$ of length $n$ with entries in $A$.
Write $LEV(C_\phi)$ and $REV(C_\phi)$ for the sets of left and right eigenvalues of $C_\phi$. 

\begin{thm}\label{Comp}
For any standard polynomial $\phi(z)$ over $A$, $LEV(C_\phi)=\bigcup_{\gamma \in A^\times} R(\phi_\gamma)$.  
\end{thm}

\begin{proof}
The element $\lambda \in A$ is a left eigenvalue of $C_\phi$ if and only if there exists a nonzero vector 
$$v=\left(\begin{array}{r}
v_1\\
\vdots\\
v_n
\end{array}\right) \in A^n
$$ satisfying $C_\phi v=\lambda v$.
This equality is equivalent to the system
\begin{eqnarray*}
v_2 & = & \lambda v_1\\
\vdots\\
v_n & = & \lambda v_{n-1}\\
-c_0 v_1-\dots-c_{n-1} v_n & = & \lambda v_n.
\end{eqnarray*}
Note that since $v\ne0$, $v_1\ne0$.
The first $n-1$ equations mean that $v$ is $\left(\begin{array}{r} 1 \\ \lambda \\ \vdots\\ \lambda^{n-1} \end{array}\right)v_1$ and the last equation then becomes 
$$c_0 v_1+c_1 (\lambda v_1)+\dots+c_{n-1} (\lambda^{n-1} v_1)+\lambda^n v_1=0.$$
Write $\gamma=v_1$.
Note that if $\gamma=0$ then $v$ is the zero vector, so we have $\gamma \neq 0$.
For each $i\in \{0,\dots,n-1\}$, write $c_i'=\gamma^{-1} c_i$, and so the equation becomes
$$\gamma c_0' \gamma+(\gamma c_1') (\lambda \gamma)+\dots+(\gamma c_{n-1}') (\lambda^{n-1} \gamma)+\lambda^n \gamma=0.$$
By the Moufang identity $(xy)(zx)=x(yz)x$, the former equation becomes
$$\gamma (c_0'+c_1'\lambda+\dots+c_{n-1}'\lambda^{n-1} +\gamma^{-1} \lambda^n) \gamma=0.$$
Therefore, $\lambda$ is a root of the twisted polynomial $\phi_\gamma(z)$.
In the opposite direction, it is clear from the same computation that a root $\lambda$ of $\phi_\gamma(z)$ is in $LEV(C_\phi)$.
\end{proof}

\begin{thm}
Let $\phi(z)$ be a standard monic polynomial over an octonion algebra $A$ over a field $F$, and let $E(N,T)$ and $G(N,T)$ be as in Theorem \ref{classes}.
Then
\begin{enumerate}
\item $R(\phi) \subseteq LEV(C_\phi) \subseteq R(\Phi)$.
\item For every conjugacy class $[z_0] \in R(\Phi)$ of norm $N_0$ and trace $T_0$, if $E(N_0,T_0)=0$ then $[z_0] \subseteq LEV(C_\phi)$. Otherwise, \\$[z_0] \cap LEV(C_\phi)=\{-(E(N_0,T_0)^{-1} \gamma)(\gamma^{-1} G(N_0,T_0)) : \gamma \in A^\times\}$.
\end{enumerate}
\end{thm}

\begin{proof}
The inclusion $R(\phi) \subseteq LEV(C_\phi)$ is obvious.
For $LEV(C_\phi) \subseteq R(\Phi)$, it is enough to notice that the twists $\phi_\gamma(z)$ have the same companion polynomial as $\phi(z)$ up to division by the norm of $\gamma$ (using the multiplicativity of the norm form and \cite[Equations (1.3) \& (1.4), Section 1.2]{SpringerVeldkamp}).

For each $\gamma$, the polynomial $\phi_\gamma(z)$ satisfies $\phi_\gamma(z)=(\gamma^{-1} E(N,T))z+(\gamma^{-1} G(N,T))$ by a straight-forward computation.
Consider a given class $[z_0]$ in $\Phi(z)$ of norm $N_0$ and trace $T_0$.
If $E(N_0,T_0)=0$ then $[z_0] \subseteq R(\phi)$, and hence $[z_0] \subseteq LEV(C_\phi)$.
Suppose $E(N_0,T_0) \neq 0$. Then the unique element of $R(\phi_\gamma) \cap [z_0]$ is \\$-(\gamma^{-1} E(N_0,T_0))^{-1}(\gamma^{-1} G(N_0,T_0))=-(E(N_0,T_0)^{-1} \gamma)(\gamma^{-1} G(N_0,T_0))$.
\end{proof}

Note that unlike the case of quaternion algebras, there is no inclusion $LEV(C_\phi) \subseteq R(\phi)$, not even in the case of quadratic polynomials.

\begin{exmpl}\label{LEVnotRoots}
Consider the polynomial $\phi(z)=z^2+iz+1+ij$.
The element $\lambda=j$ is not a root of this polynomial.
However, $\lambda=j$ is a root of the twisted polynomial
$\phi_\ell(z)$, and so it belongs to $LEV(C_\phi)$.
\end{exmpl}

Note that in this example, $j$ belongs to the quaternion subalgebra generated by the coefficients, which means that even in the case where all the coefficients belong to the same quaternion subalgebra $Q$, there is no guarantee that $LEV(C_\phi) \cap Q=R(\phi) \cap Q$.
The following proposition describes the set $LEV(C_\phi) \cap Q$ in such cases:

\begin{prop}
Given a standard monic polynomial $\phi(z)=z^n+c_{n-1} z^{n-1}+\dots+c_1 z+c_0$ over a division octonion algebra $A$ whose coefficients belong to a quaternion subalgebra $Q$ of $A$, we have $LEV(C_\phi) \cap Q=(R(\phi)\cup R(\tilde{\phi})) \cap Q$.
\end{prop}

\begin{proof}
Every left eigenvalue $\lambda$ of $C_\phi$ satisfies
$$c_0 \gamma+c_1 (\lambda \gamma)+\dots+c_{n-1} (\lambda^{n-1} \gamma)+\lambda^n \gamma=0$$
for some $\gamma \in A^\times$.
Suppose all the coefficients belong to a quaternion subalgebra $Q$, and suppose $\lambda \in Q$ as well.
Then $A$ decomposes as $A=Q\oplus Q\ell$.
The element $\gamma$ decomposes accordingly as $\gamma=\gamma_0+\gamma_1 \ell$.
By a straight-forward computation, we obtain
$c_0 \gamma+c_1 (\lambda \gamma)+\dots+c_{n-1} (\lambda^{n-1} \gamma)+\lambda^n \gamma=(c_0+c_1 \lambda+\dots+c_{n-1} \lambda^{n-1}) \gamma_0+(\gamma_1 (c_0+\lambda c_1+\dots+\lambda^{n-1} c_{n-1}+\lambda^n)) \ell.$
Therefore, 
$$(\gamma_0=0 \vee c_0+c_1 \lambda+\dots+c_{n-1} \lambda^{n-1}=0)\wedge(\gamma_1=0 \vee  c_0+\lambda c_1+\dots+\lambda^{n-1} c_{n-1}+\lambda^n=0).$$
Consequently, $LEV(C_\phi) \cap Q \subseteq (R(\phi)\cup R(\tilde{\phi})) \cap Q$. The inclusion in the opposite direction is proven using the same computation.
\end{proof}
\section{Right Eigenvalues of the Companion Matrix}

Given a polynomial $\phi(z)=z^n+c_{n-1} z^{n-1}+\dots+c_1 z+c_0$, let $\phi^\gamma(z)$ denote the two-sided twisted polynomial $\phi^\gamma(z)=\gamma^{-2} z^n+(\gamma^{-1} c_{n-1} \gamma^{-1}) z^{n-1}+\dots+(\gamma^{-1} c_1 \gamma^{-1}) z+\gamma^{-1} c_0 \gamma^{-1}$.
\begin{thm}\label{REVT}
The set $REV(C_\phi)$ is the union of $R(\phi^\gamma)$ for all $\gamma \in A^\times$.
\end{thm}

\begin{proof}
The element $\lambda \in A$ is a right eigenvalue of $C_\phi$ if and only if there exists a nonzero vector 
$$v=\left(\begin{array}{r}
v_1\\
\vdots\\
v_n
\end{array}\right) \in A^n
$$ satisfying $C_\phi v=v \lambda$.
This equality is equivalent to the system
\begin{eqnarray*}
v_2 & = & v_1 \lambda\\
\vdots\\
v_n & = & v_{n-1} \lambda\\
-c_0 v_1-\dots-c_{n-1} v_n & = & v_n\lambda.
\end{eqnarray*}
Note that since $v\ne0$, $v_1\ne0$.
The first $n-1$ equations mean that $v$ is $v_1 \left(\begin{array}{r} 1 \\ \lambda \\ \vdots\\ \lambda^{n-1} \end{array}\right)$ and the last equation then becomes 
$$c_0 v_1+c_1 (v_1\lambda)+\dots+c_{n-1} (v_1\lambda^{n-1})+v_1\lambda^n=0.$$
Note that if $v_1=0$ then $v$ is the zero vector, so we have $v_1 \neq 0$. Write $\gamma=v_1^{-1}$.
Multiply the equation from the right by $\gamma^{-1}$ and use the Moufang identity $x(y(xz))=(xyx)z$ to get
$$\gamma^{-1} c_0 \gamma^{-1}+(\gamma^{-1} c_1 \gamma^{-1}) \lambda+\dots+(\gamma^{-1} c_{n-1} \gamma^{-1}) \lambda^{n-1}+\gamma^{-2} \lambda^n=0.$$
Therefore, $\lambda$ is a root of the twisted polynomial $\phi^\gamma(z)$.
In the opposite direction, it is clear from the same computation that a root $\lambda$ of $\phi^\gamma(z)$ is in $REV(C_\phi)$.
\end{proof}

\begin{thm}\label{REV}
Let $\phi(z)$ be a standard monic polynomial over an octonion algebra $A$ over a field $F$, and let $E(N,T)$ and $G(N,T)$ be as in Theorem \ref{classes}.
Then
\begin{enumerate}
\item $R(\phi) \subseteq REV(C_\phi) \subseteq R(\Phi)$.
\item \sloppy For every conjugacy class $[z_0] \in R(\Phi(z))$ of norm $N_0$ and trace $T_0$, if $E(N_0,T_0)=0$ then $[z_0] \subseteq REV(C_\phi)$. Otherwise, $[z_0] \cap REV(C_\phi)=\{-\gamma \big(E(N_0,T_0)^{-1} (G(N_0,T_0) \gamma^{-1})\big) : \gamma \in A^\times\}$.
\end{enumerate}
\end{thm}

\begin{proof}
The inclusion $R(\phi) \subseteq REV(C_\phi)$ is obvious.
For $REV(C_\phi) \subseteq R(\Phi)$, it is enough to notice that the twists $\phi^\gamma(z)$ have the same companion polynomial as $\phi(z)$ up to division by the norm of $\gamma^2$ (using the multiplicativity of the norm form and \cite[Equations (1.3) \& (1.4), Section 1.2]{SpringerVeldkamp}).

\sloppy For each $\gamma$, the polynomial $\phi^\gamma(z)$ satisfies $\phi^\gamma(z)=(\gamma^{-1} E(N,T) \gamma^{-1})z+(\gamma^{-1} G(N,T) \gamma^{-1})$ by a straight-forward computation.
Consider a given class $[z_0]$ in $R(\Phi)$ of norm $N_0$ and trace $T_0$.
If $E(N_0,T_0)=0$ then $[z_0] \subseteq R(\phi)$, and hence $[z_0] \subseteq REV(C_\phi)$.
Suppose $E(N_0,T_0) \neq 0$. Then the unique element of $R(\phi_\gamma) \cap [z_0]$ is $-(\gamma^{-1} E(N_0,T_0) \gamma^{-1})^{-1}(\gamma^{-1} G(N_0,T_0) \gamma^{-1})$.
By the Moufang identity $(xyx)z=x(y(xz))$ we obtain $-(\gamma^{-1} E(N_0,T_0) \gamma^{-1})^{-1}(\gamma^{-1} G(N_0,T_0) \gamma^{-1})=-\gamma \big(E(N_0,T_0)^{-1} (G(N_0,T_0) \gamma^{-1}))$.
\end{proof}

\begin{rem}
Given an octonion algebra $A$ over a field $F$, if $g,h \in A$ are conjugates (i.e., have the same trace and norm), then $g=\delta h \delta^{-1}$ for some $\delta \in A^\times$ of $\tr(\delta)=0$.
\end{rem}

\begin{proof}
If $g \neq \overline{h}$, take $\delta=g-\overline{h}$.
If $g=\overline{h}$, take $\delta$ to be any element in $1^{\perp} \cap g^{\perp}$.
\end{proof}

\begin{thm}\label{conjugacy}
Let $A$ be an octonion division algebra over a field $F$, and let $e$ and $g$ be nonzero elements in $A$.
Then $\{\gamma (e (g \gamma^{-1})) : \gamma \in A^\times\}=\{\delta (eg) \delta^{-1} : \delta \in A^\times\}$.
\end{thm}

\begin{proof}
The left-to-right inclusion follows from the fact that every element of the form $\gamma (e (g \gamma^{-1}))$ has the same trace and norm as $eg$:
The norm is multiplicative, so it follows immediately that $\norm(\gamma (e (g \gamma^{-1})))=\norm(eg)$.
The trace of $\gamma (e (g \gamma^{-1}))$ is $\langle 1,\gamma (e (g \gamma^{-1})) \rangle$ where $\langle \ , \rangle$ is the polarization of the norm form.
Now,
$$\langle 1,\gamma (e (g \gamma^{-1})) \rangle=\langle \overline{\gamma}, e (g\gamma^{-1}) \rangle=\langle \overline{e}\overline{\gamma}, g\gamma^{-1} \rangle=\langle (\overline{e} \overline{\gamma}) \overline{\gamma}^{-1},g \rangle=\langle \overline{e},g \rangle=\langle 1,eg \rangle=\tr(eg).$$
The computation makes use of the well-known identity $\langle xy,z \rangle=\langle y,\overline{x}z \rangle=\langle x,z\overline{y} \rangle$ that can be found in \cite[Lemma 1.3.2]{SpringerVeldkamp}.

For the opposite inclusion, we note that $eg$ and $ge$ are conjugates, so every element in the right set can be written as $\delta (ge) \delta^{-1}$ for some $\delta \in A^\times$ of $\tr(\delta)=0$. Set $\gamma=\delta g$, and then $\delta^{-1}=g \gamma^{-1}$. Now, $\delta^{-1}$ is a scalar multiple of $\delta$ (because $\tr(\delta)=0$), and by the Moufang identity $(xyx)z=x(y(xz))$ we obtain
$$\delta (ge) \delta^{-1}=(\delta g)(e \delta^{-1})=\gamma (e (g \gamma^{-1})).$$
\end{proof}

\begin{cor}
Given a standard monic polynomial $\phi(z)$ over an octonion division algebra $A$ over a field $F$ with companion polynomial $\Phi(z)$, we have $R(\Phi)=REV(C_\phi)$.
\end{cor}

\begin{proof}
\sloppy By Theorem \ref{REV}, for each conjugacy class $[z_0]$ in $R(\Phi)$, either the entire conjugacy class is in $R(\phi)$ (which happens when $E(N_0,T_0)=0$) and then it is also in $REV(C_\phi)$, or the intersection with $REV(C_\phi)$ is of the form $\{-\gamma \big(E(N_0,T_0)^{-1} (G(N_0,T_0) \gamma^{-1})\big) : \gamma \in A^\times\}$ (which happens when $E(N_0,T_0) \neq 0$). By Theorem \ref{conjugacy}, this set is the conjugacy class of $-E(N_0,T_0)^{-1} G(N_0,T_0)$, which is $[z_0]$.
\end{proof}

\section*{Acknowledgements}
The author is indebted to Seidon Alsaody for many illuminating discussions concerning this paper and other problems related to octonion algebras. 
The author was visiting Perimeter Institute for Theoretical Physics in the Summer of 2018, during which a major part of this project was carried out.

\end{document}